\documentclass[12pt]{article}
\usepackage{amsmath}
\usepackage{amsfonts}
\usepackage{amssymb}
\usepackage{amsthm}
\usepackage{bigints}
\usepackage{amstext}
\usepackage{color}
\usepackage{graphicx}

\usepackage{color}

\newtheorem{theorem}{Theorem}

\newtheorem{corollary}[theorem]{Corollary}

\newtheorem{lemma}[theorem]{Lemma}

\newtheorem{proposition}[theorem]{Proposition}
\newtheorem{remark}[theorem]{Remark}

\def\div{\mbox{div}\,}

\def\curl{\mbox{curl}\,}

\def\D{\mathbb{D}}

\def\R{\mathbb{R}}

\def\Sc{\mbox{Sc}\,}

\def\tr{\operatorname{\rm tr}}

\begin{document}

\title{Hodge decomposition for generalized Vekua spaces in higher dimensions}
\author{Briceyda B. Delgado\thanks{INFOTEC, Centro de Investigación e Innovación en Tecnologías de la Información y Comunicación, Cto. Tecnopolo Sur No. 112, Pocitos, Aguascalientes 20326, Mexico, briceydadelgado@gmail.com}}
\date{}
\maketitle
\begin{abstract}
 We introduce the spaces $A^p_{\alpha, \beta}(\Omega)$ of $L^p$-solutions to the Vekua equation (generalized monogenic functions) $D w=\alpha\overline{w}+\beta w$ in a bounded domain in $\R^n$, where $D=\sum_{i=1}^n e_i \partial_i$ is the Moisil-Teodorescu operator, $\alpha$ and $\beta$ are bounded functions on $\Omega$. 
 
 The main result of this work consists of a Hodge decomposition of the $ L^2$ solutions of the Vekua equation. From this orthogonal decomposition arises an operator associated with the Vekua operator, which in turn factorizes certain Schrödinger operators.
Moreover, we provide an explicit expression of the ortho-projection over $A^p_{\alpha, \beta}(\Omega)$ in terms of the well-known ortho-projection of $L^2$ monogenic functions and an isomorphism operator. Finally, we prove the existence of component-wise reproductive Vekua kernels and the interrelationship with the Vekua projection in Bergman's sense. 
\end{abstract}
\mbox{}spaces, Schrödinger equation.\\ 
\mbox{} 


\section{Introduction}
The Vekua equation was proposed by Ilya Vekua in the 1950s. Although the Vekua equation was proposed more than six decades ago, there are still many open problems related to this equation.
One of the most important open problems is the search for explicit solutions to the Vekua equation. In many cases, it is not possible to find explicit solutions to this equation, which hinders its practical application in solving physical and mathematical problems.
Another important problem is the study of the properties of the solutions of the Vekua equation and the establishment of which properties are inherited from the analytical (monogenic) functions. 
The foundations of the general theory of generalized analytic functions in the plane were developed simultaneously by Ilya Vekua in his book ``Generalized analytic functions" in \cite{Vekua1962} and by Lipman Bers in his monograph ``Theory of pseudo-analytic functions" \cite{Bers1953}, although Bers called them pseudo-analytic functions, they are basically the same as the so-called generalized analytic functions of Vekua. 

It is important to mention that although a Similarity Principle is not known as it is in the complex case, the tools used in this work rely on the properties of an isomorphism operator, which sends solutions of the Vekua equation into monogenic functions. 
In this work, we will exploit the Hodge decomposition of monogenic functions in order to generalize the orthogonal decomposition to the space of $L^2$ solutions of the Vekua equation
\begin{equation}\label{eq:Vekua-intro}
  Dw=\alpha \overline{w}+\beta w,
\end{equation}
where $D=\sum_{i=1}^n e_i\partial_i$ is the $n$-dimensional Moisil-Tedorescu operator. The importance of the study of the Hodge decomposition and the Vekua equation \eqref{eq:Vekua-intro} lies in its applications. For instance, to associated boundary values problems, integral representations of the solutions, and inverse problems. For the particular case of logarithmic derivative coefficients, i.e. $\alpha=\nabla f/f$ and $\beta=\nabla g/g$, the scalar part of a solution of the Vekua equation is a solution of a conductivity equation, which is fundamental in the Calderon's inverse problem. This relationship has been widely explored in the complex case, and there are numerous works that highlight this correspondence between their solutions, see for instance,  \cite{BLRR2010, BFL2011, Krav2009} and references therein. Concerning inverse problems we suggest \cite{AstPai2006, Fischer-1, Fischer-2, Fischer-3} and for associated boundary value problems \cite{Bers1954, Vekua1962}.

In higher dimensions, particularly in the Clifford algebra setting, in the 1990s Malonek extended the concept of $(F, G)$ derivative in the sense of Bers to quaternion-valued functions and derived a spatial version of pseudo-analytic functions \cite{Malonek1996,Malonek1998}. Later in \cite{Berglez2008}, Berglez showed a representation formula for the solutions of the biquaternionic Vekua equation $D^*w=\varphi(x_1)\overline{w}$, where $D^*=-iD$ and $D$ is the usual Dirac operator in quaternionic analysis. Recently, in \cite{DelPor2017, DelPor2018}, the author proposed a Hilbert transform for the quaternionic main Vekua equation and solved the problem of constructing the vector part of the solutions when the scalar part is known. More recently, in the bicomplex setting \cite{VictorRadial, radial2024}, the authors constructed a transmutation operator that sends bicomplex analytic functions into solutions of the radial bicomplex Vekua equation.

The outline of the paper is as follows. In Section \ref{sec:Bergman-Vekua spaces}, we introduce the \textit{generalized Vekua spaces}, also known as the space of generalized monogenic functions, in the sense of Vekua. Moreover, we analyze the properties satisfied by the continuous operator $S_{\alpha, \beta}=I-T_{\Omega}[\alpha C+\beta I]$ which transforms solutions of the Vekua equation into monogenic functions in $\Omega$.
In Section \ref{sec:Hodge}, we provide the following Hodge decomposition
\begin{equation}\label{eq:deco-intro}
   L^2(\Omega,C\ell_{0,n})=\text{Ker } (D-\alpha C-\beta)\cap L^2(\Omega,C\ell_{0,n})\oplus (D-M^{\alpha}C-\overline{\beta}) W_0^{1,2}(\Omega,C\ell_{0,n}),
\end{equation}
where $C$ denotes the conjugation operator in $C\ell_{0,n}$, $M^{\alpha}$ is the right-side multiplication operator and the orthogonality evolved in \eqref{eq:deco-intro} is under the scalar product $\langle u,v\rangle=\text{Sc }\int_{\Omega}\overline{u}v$. 
The above Hodge decomposition \eqref{eq:deco-intro} generalizes the Hodge decompositions found in the pioneering works of Spr\"{o}ssig for the kernel of Vekua-type operator with $\alpha\equiv 0$ and $\beta=\sum_{i=0}^n \beta_ie_i$ be a paravector in $C\ell_{0,n}$, see for instance \cite{Sprossig1995, Sprossig2001}. See also \cite{Sprossig1993} for some decompositions of generalized monogenic functions in the context of Clifford algebras (when $\alpha=0$,  $\beta=\overline{\partial} p/p$ in \eqref{eq:Vekua_equation}, and $p$ be a positive $C^{\infty}$ scalar function).
To the best of our knowledge, the orthogonal decomposition \eqref{eq:deco-intro} was not known for the Vekua equation with $\alpha\ne 0$, even in the complex case. 

During this work, we will consider the operator $D-M^{\alpha}C-\overline{\beta}$ appearing in the above orthogonal decomposition \eqref{eq:deco-intro} as the \textit{adjoint} of the Vekua operator $D-\alpha C-\beta$, restricted to a subspace of $L^2(\Omega, C\ell_{0,n})$ with zero vanishing trace. Moreover, we factorize two different Schrödinger operators through these operators as follows
\begin{align}\label{eq:factorizations-intro}
    \nonumber \left(-\Delta+|\alpha|^2+|\beta|^2+\div \vec \alpha-\div \vec \beta+2\alpha\cdot \overline{\beta}\right)h_0&=\text{Sc } (D-\alpha C-\beta)(D-M^{\alpha}C-\overline{\beta})h_0,\\
    \text{ }\\
    \nonumber\left(-\Delta+|\alpha|^2+|\beta|^2+\div \vec \alpha+\div \vec \beta+2\alpha\cdot \beta\right)h_0&=\text{Sc } (D-M^{\alpha}C-\overline{\beta})(D-\alpha C-\beta)h_0,
\end{align}
where $h_0$ is a scalar function.
In \cite{Bernstein1996, Bernstein1999}, Bernstein et al. established a factorization of the Schrödinger operator in terms of the operators $D\pm M^{\gamma}$ in a way that requires a solution of an associated Riccati equation in the three-dimensional case. Later, in \cite{KravKrav2001}, Kravchenko showed that $\gamma=\nabla f/f$ and the factorization reduces to
\begin{equation}\label{eq:Schrodinger-intro}
   \left(-\Delta+\dfrac{\Delta f}{f}\right)u_0=(D+M^{\nabla f/f})(D-M^{\nabla f/f})u_0.
\end{equation}
In the particular case, $\alpha=\nabla f /f$ and $\beta=0$, then both left-hand sides of \eqref{eq:factorizations-intro} reduces to the Schrödinger equation of the left-hand side of \eqref{eq:Schrodinger-intro}.

In Section \ref{sec:Bergman-projection}, we construct an ortho-projection $P_{\alpha,\beta}$ from $L^2(\Omega, C\ell_{0,n})$ to the generalized Vekua space $A^2_{\alpha,\beta}(\Omega, C\ell_{0,n})$, using the isomorphism operators $S_{\alpha,\beta}$ and $S_{\alpha,\beta}^*$ as well as the ortho-projection generated by the Hodge decomposition of monogenic functions, namely $P_{\Omega}$. More precisely, we prove that
\begin{equation}
    P_{\alpha,\beta}=S_{\alpha,\beta}^*\, P_{\Omega}\,(S_{\alpha,\beta}^*)^{-1}=S_{\alpha,\beta}^{-1} \,P_{\Omega} \,S_{\alpha,\beta}.
\end{equation}
In Subsection \ref{sec:kernel}, we extend the results provided in \cite{Campos2020} for the complex case. More precisely, we establish the existence of component-wise reproductive Vekua kernels, and we provide another expression for the Vekua projection satisfying a component-wise reproductive property (see Proposition \ref{prop:projector})
\begin{equation}\label{eq:reproductive-intro}
P_{\alpha,\beta}[w](\vec x)=\bigintsss_{\Omega} \sum_B (-1)^{\frac{|B|(|B|+1)}{2}}\, e_B^2\,  K^B_{\alpha,\beta}(\vec x,\vec y) \, w_B(\vec y) \,d\sigma_{\vec y}, \qquad \forall w\in L^2(\Omega,C\ell_{0,n}). 
\end{equation}



\section{Preliminaries}\label{subsec:Preliminaries}
Let us consider the real Clifford algebra $C\ell_{0,n}$ generated by the elements $e_0=1,e_1,e_2,\cdots,e_n$ with the relation $e_ie_j+e_je_i=-2\delta_{ij}$, where $\delta_{i,j}$ is the Kronecker delta function $(i,j=1,2,\cdots,n)$. We embed the Euclidean space $\R^{n+1}$ in $C\ell_{0,n}$ through the identification of $x=(x_0,x_1,x_1,\cdots,x_n)\in \R^{n+1}$ with the Clifford \textit{paravector} $x=x_0+\sum_{i=1}^n x_ie_i$. Let us denote by $\text{Sc }x=x_0$ and $\text{Vec }x=\vec x=\sum _{i=1}^n e_ix_i \in \R^n$ the scalar and the vector part of any arbitrary paravector $x$. From now on, $\Omega\subset \R^n$ will be a bounded domain with sufficiently smooth boundary $\partial\Omega$, and the elements $\vec x\in \Omega$ will be called \textit{vectors}.
A basis for $C\ell_{0,n}$ is the subset
\[
  \{e_0=1, e_A=e_{i_1}e_{i_2}\cdots e_{i_k} \colon A=\{i_1,i_2,\cdots i_k\}, 1\leq i_1<i_2<\cdots <i_k\leq n\}.
\]
We define the \textit{conjugation} in $C\ell_{0,n}$ as usual $\overline{ab}=\overline{b}\overline{a}$, $\forall a,b\in C\ell_{0,n}$. We will say that $w$ is a $C\ell_{0,n}$-valued function defined in $\Omega$, $w \colon \Omega \to C\ell_{0,n}$, if
\[ 
w(\vec x)=\sum_A w_A(\vec x)\, e_A, \quad \vec x\in \Omega,
\]
where the coordinates $w_A$ are real-valued functions defined in $\Omega$, that is $w_A\colon \Omega \to \R$. In particular, paravector-valued functions are denoted by $g(\vec x)=\sum_{i=0}^n g_i(\vec x) e_i$. 

In \cite{BDS1982} was proved several generalizations of classic results of functional analysis for right (left) $C\ell_{0,n}$-modules. For instance, the Riesz representation theorem \cite[Th.\ 7.6]{BDS1982}, the Hahn-Banach extension theorem \cite[Th.\ 2.10]{BDS1982}, among others.
Let us denote by $L^p(\Omega,C\ell_{0,n})$ the set of $C\ell_{0,n}$-valued functions defined in $\Omega$ which are $p$-integrable (that is, component-wise $p$-integrable) $1<p<\infty$. It is easy to check that $L^p(\Omega,C\ell_{0,n})$ is a right $C\ell_{0,n}$-module. In particular, $L^2(\Omega,C\ell_{0,n})$ is a right Hilbert $C\ell_{0,n}$-module under the inner product
\begin{align}\label{eq:inner-product}
  \langle u,v\rangle_{L^2(\Omega)}=\int_{\Omega}\overline{u} v \, d\sigma, \quad \forall u,v\in L^2(\Omega,C\ell_{0,n}).
\end{align}
Some properties satisfied by this inner product are 
\begin{enumerate}
\item $\langle u,v a+w \rangle_{L^2(\Omega)}=\langle u,v \rangle_{L^2(\Omega)}a + \langle u, w \rangle_{L^2(\Omega)}$,
\item $\langle u,v \rangle_{L^2(\Omega)}=\overline{\langle v,u \rangle}_{L^2(\Omega)}$,
\end{enumerate}
for all $u, v, w \in L^2(\Omega,C\ell_{0,n})$ and $a\in C\ell_{0,n}$. Other functional spaces that will be used throughout this paper are the classic Sobolev spaces $W^{1,p}(\Omega, C\ell_{0,n})$, its subspace with vanishing trace $W_0^{1,p}(\Omega, C\ell_{0,n})$ and the space of test functions $C_{c}^{\infty}(\Omega, C\ell_{0,n})$. If $\Omega$ is a bounded Lipschitz domain, then the trace operator $\text{tr}\colon W^{1,p}(\Omega,C\ell_{0,n})\rightarrow L^p(\Omega,C\ell_{0,n})$, defined by $\text{tr }u=u|_{\partial\Omega}$ is continuous. 
  
The \textit{Moisil-Teodorescu} differential operator is defined as follows
\begin{align}\label{eq:Moisil-Teodorescu}
  D=\sum_{i=1}^n e_i \, \partial_i,
\end{align}
where $\partial_i$ means the partial derivative with respect to the variable $x_i$. We will say that $w\in C^1(\Omega,C\ell_{0,n})$ is \textit{monogenic} if $D\, w=0$ in $\Omega$. By the relation $\Delta=-D^2$ and Weyl's Lemma, we will not make a distinction between weak and strong monogenic functions.

Let us denote by $A^p(\Omega,C\ell_{0,n})$ the subspace of monogenic functions in $L^p(\Omega,C\ell_{0,n})$, $1<p<\infty$, which is also a right $C\ell_{0,n}$-module, with the usual norm given by
\begin{align*}
   \|w\|_{A^p(\Omega)}^p=\|w\|_{L^p(\Omega)}^p=\int_{\Omega}{|w|^p \, d\sigma}.
\end{align*}
Due to $A^p(\Omega,C\ell_{0,n})$ is a closed subspace in $L^p(\Omega, C\ell_{0,n})$ $1<p<\infty$ \cite[Prop.\ 3.73]{GuSpr1997}.
Therefore, the right Hilbert $C\ell_{0,n}$-module $L^2(\Omega, C\ell_{0,n})$ allows an orthogonal decomposition
\begin{lemma}\label{lemma:Hodge-monogenic}
\cite[Th.\ 1]{GurSpr1995}The Hilbert space $L^2(\Omega, C\ell_{0,n})$ allows an orthogonal decomposition
\begin{align}\label{eq:Hodge-1}
   L^2(\Omega, C\ell_{0,n})=A^2(\Omega,C\ell_{0,n})\oplus D W_0^{1,2}(\Omega,C\ell_{0,n}),
\end{align} 
where $A^2(\Omega,C\ell_{0,n})=\text{Ker }D\cap L^2(\Omega, C\ell_{0,n})$ and with respect to the inner product \eqref{eq:inner-product}.
\end{lemma}
Later, in Theorem \ref{th:Hodge-decomposition}, we will generalize the above Hodge decomposition for the space of $L^2$-solutions of the Vekua equation.

We define the \textit{Teodorescu transform} \cite{GuHaSpr2008}
\begin{align}\label{eq:Teodorescu_operator}
   T_{\Omega}[w](\vec x)=-\int_{\Omega} 
	  E(\vec y-\vec x) w(\vec y)\, d\vec y, \quad \vec x\in\R^n,
\end{align}
where $w\in L^p(\Omega,C\ell_{0,n})$ and the Cauchy kernel $E$ is given by
\begin{equation}\label{eq:Cauchy-kernel}
E(\vec x)=-\frac{\vec x}{\sigma_n|\vec x|^n}, \quad \vec x\in \R^n\setminus\{0\},
\end{equation}
and $\sigma_n$ is the surface area of the unit sphere in $\R^n$. Moreover, $T_{\Omega}\colon L^p(\Omega, C\ell_{0,n})\to L^p(\Omega, C\ell_{0,n})$ is a bounded operator, and it is a right inverse of the Moisil-Teodorescu operator $D$. That is, $DT_{\Omega}=I$ in $\Omega$.
For $w\in L^p(\partial\Omega)$, the \textit{Cauchy operator} is defined by
\begin{align}\label{eq:Cauchy_operator}
   F_{\partial\Omega}[w](\vec x)=\int_{\partial\Omega} 
	  E(\vec y-\vec x) \eta (\vec y)w(\vec y)\, ds_{\vec y}, \quad \vec x\in\R^n\setminus \partial\Omega,
\end{align}
where $\vec \eta$ is the outward normal vector to $\partial\Omega$.

Besides, the ortho-projection $P_{\Omega} \colon L^2(\Omega,C\ell_{0,n}) \to A^2(\Omega,C\ell_{0,n})$ has the explicit form
\begin{align}\label{eq:Bergman_projection}
   P_{\Omega}=F_{\partial\Omega} (\tr T_{\Omega} F_{\partial\Omega})^{-1} \tr T_{\Omega},
\end{align}
where $T_{\Omega}$ and $F_{\partial\Omega}$ are the Teodorescu transform \eqref{eq:Teodorescu_operator} and the Cauchy operator \eqref{eq:Cauchy_operator}, respectively.


\section{Generalized Vekua spaces}\label{sec:Bergman-Vekua spaces}

Let $1<p<\infty$ and let $\alpha,\beta \in L^{\infty}(\Omega,C\ell_{0,n})$ be bounded functions on $\Omega$. We are interested in the analysis of solutions of the \textit{Vekua equation}
\begin{align}\label{eq:Vekua_equation}
   D w=\alpha\, \overline{w}+\beta w, \quad \text{ in }\Omega.
\end{align}

Let us define the \textit{generalized Vekua space} $A_{\alpha,\beta}^p(\Omega,C\ell_{0,n})$ to consist of $C\ell_{0,n}$-valued functions $w \in L^p(\Omega,C\ell_{0,n})$ satisfying \eqref{eq:Vekua_equation} in the sense of distributions on $\Omega$. 
For simplicity, sometimes this subspace will be denoted by $A_{\alpha,\beta}^p(\Omega)=A_{\alpha,\beta}^p(\Omega,C\ell_{0,n})$. The equation \eqref{eq:Vekua_equation} was called the \textit{Vekua equation} in \cite{Vekua1962} in the complex case. 


Notice that if $\alpha\equiv 0$, then $A_{0,\beta}^p(\Omega,C\ell_{0,n})$ is a subspace over $C\ell_{0,n}$. Meanwhile, if $\alpha\ne 0$, due to the Clifford conjugation acting in \eqref{eq:Vekua_equation} we have that $A_{\alpha,\beta}^p(\Omega,C\ell_{0,n})$ is a subspace over $\R$. 

Moreover, the complex version of the Vekua equation \eqref{eq:Vekua_equation} 
 has been analyzed in many works, see for instance \cite{BLRR2010, BFL2011, Fischer-1, Fischer-2, Campos2020, DelLeb2020}. In this case, $w=w_0+iw_1$ satisfies $\overline{\partial}w=\alpha \overline{w}$ if and only if $w$ satisfies the following system of generalized Cauchy-Riemann equations
\begin{align*} 
   \nonumber \partial_x(\sigma w_1)=-\sigma^2 \partial_y \left(\frac{w_0}{\sigma}\right),   \quad
	  \partial_y(\sigma w_1)=\sigma^2 \partial_x\left(\frac{w_0}{\sigma}\right),
\end{align*}
where $\alpha=\overline{\partial}\log \sigma$, $\partial_x$ and $\partial_y$ are the partial derivatives with respect to $x$ and $y$, respectively.

Meanwhile, in more dimensions, this equation \eqref{eq:Vekua_equation} has been little explored. In the quaternionic case, one of the pioneering works is \cite[Ch.\ 16]{Krav2009}, for the quaternionic case we mean quaternionic valued functions defined in a three-dimensional domain. It was established in  \cite{Krav2009} that $w=w_0+\vec w=w_0+\sum_{i=1}^3 w_i e_i$ is a solution of the Vekua equation $Dw=(Df/f)\overline{w}$ if and only if
\[
 \div (f\vec w)=0, \quad 
 \curl( f\vec w)=-f^2 \nabla\left(f^{-1}w_0\right), \quad \text{ in }\Omega.
\]
In \cite{DelPor2019}, an important relation between the Vekua equation and the Beltrami equation was established. More precisely, if $\alpha=\nabla f/f$ and $\beta=\nabla g/g$, where $f$ and $g$ are non-vanishing scalar-valued functions in $W^{1,\infty}(\Omega)$, we have that $w=w_0+\vec w$ is a solution of the quaternionic Vekua equation 
\begin{equation}\label{eq:Vekua-quaternionic}
  Dw=\dfrac{\nabla f}{f} \overline{w}+\dfrac{\nabla g}{g}w,
\end{equation}
if and only if $u=w_0/(fg)+(f\vec w)/g$ is solution of the \textit{quaternionic Beltrami equation} \cite{DelPor2019}
\begin{equation}\label{eq:Beltrami-equation}
   Du=\dfrac{1-f^2}{1+f^2} D\overline{u}.
\end{equation}
A similar relation was established in \cite{Santa2019} in order to analyze the Calderón problem in three and higher dimensions. In order to consider a similar Clifford extension of the above equivalence, we will decompose the $C\ell_{0,n}$-valued solutions of the Vekua equation as follows
\[
w=[w]_0+[w]_1+[w]_2+\cdots+[w]_n,
\]
where $[w]_k=\sum_{|A|=k} w_A\, e_A$. Equivalently, we can rewrite it as 
\[
w=\sum_{k\equiv 0,3 (\text{mod } 4)} [w]_k+\sum_{k\equiv 1,2 (\text{mod } 4)} [w]_k, \quad 0\leq k\leq n.
\]


Now, we will see that there exists an equivalence with a Clifford Beltrami equation, thus generalizing the results obtained in \cite{Santa2019, DelPor2019}.
\begin{proposition}\label{prop:Beltrami equivalence}
Let $\alpha=\nabla f/f$ and $\beta=\nabla g/g$ with $f, g$ non-vanishing scalar functions in $ W^{1,\infty}(\Omega)$. Then $w=\sum_{k\equiv 0,3 (\text{mod } 4)} [w]_k+\sum_{k\equiv 1,2 (\text{mod } 4)} [w]_k$ is solution of \eqref{eq:Vekua_equation} in $\Omega$ if and only if
\begin{equation}\label{eq:transformation-solution}
  u=\dfrac{1}{fg}\sum_{k\equiv 0,3 (\text{mod } 4)} [w]_k+\dfrac{f}{g}\sum_{k\equiv 1,2 (\text{mod } 4)} [w]_k,
\end{equation}\color{black}
is solution of the \text{Clifford Beltrami equation}
\begin{equation}\label{eq:Clifford-Beltrami}
    Du=\dfrac{1-f^2}{1+f^2} D\overline{u}\qquad \text{in }\Omega.
\end{equation}
\end{proposition}
\begin{proof}
By \cite[Prop. 3.8]{GuHaSpr2008}, we have that
\[
\overline{w}=\sum_{k\equiv 0,3 (\text{mod } 4)} [w]_k\,-\sum_{k\equiv 1,2 (\text{mod } 4)} [w]_k, \quad 0\leq k\leq n.
\]
The proof is straightforward only by noting that  $w=\sum_{k\equiv 0,3 (\text{mod} 4)} [w]_k+\sum_{k\equiv 1,2 (\text{mod} 4)} [w]_k$ is solution of \eqref{eq:Vekua_equation} in $\Omega$ if and only if 
\begin{align*}
 (fg)\, D\left(\dfrac{1}{fg} \sum_{k\equiv 0,3 (\text{mod } 4)} [w]_k \right)+\dfrac{g}{f} D\left(\dfrac{f}{g} \sum_{k\equiv 1,2 (\text{mod } 4)} [w]_k \right)=0\qquad \text{in }\Omega.
\end{align*}
That is, 
\begin{align}\label{eq:equivalent-Vekua}
 D\left(\dfrac{f}{g} \sum_{k\equiv 1,2 (\text{mod } 4)} [w]_k \right)=-f^2\, D\left(\dfrac{1}{fg} \sum_{k\equiv 0,3 (\text{mod } 4)} [w]_k \right)\qquad \text{in }\Omega.
\end{align}
Thus, 
\begin{align*}
Du=(1-f^2)\,D\left(\dfrac{1}{fg}\sum_{k\equiv 0,3 (\text{mod } 4)} [w]_k\right)\quad
D\overline{u}=(1+f^2)\,D\left(\dfrac{1}{fg}\sum_{k\equiv 0,3 (\text{mod } 4)} [w]_k\right),
\end{align*}
which in turn is equivalent to \eqref{eq:Clifford-Beltrami} as we desired.
\end{proof}
Taking the vector part of \eqref{eq:equivalent-Vekua} and applying the divergence operator, we get
\begin{corollary}\label{cor:conductividad-Schrodinger}
Under the same hypothesis than Proposition \ref{prop:Beltrami equivalence}. If $w$ solves \eqref{eq:Vekua_equation}, then   $w_0/(fg)$ is solution of the conductivity equation 
\begin{equation}\label{eq:conductivity}
   \div \left(f^2 \nabla\left(\dfrac{w_0}{fg}\right)\right)=0, \quad \text{ in }\Omega.
\end{equation}
Besides, if $w$ solves \eqref{eq:Vekua_equation}, then $w_0/g$ is solution of the Schrödinger equation
\begin{equation}\label{eq:Schrodinger}
   \left(\Delta-\dfrac{\Delta f }{f}\right)\dfrac{w_0}{g}=0, \quad \text{ in }\Omega.
\end{equation}
\end{corollary}

If $\beta\equiv 0$ and $\alpha=\nabla f/f$, then $A_{\alpha,0}^p((\Omega,C\ell_{0,n})\ne \emptyset$ because of $f\in A_{\alpha,0}^p((\Omega,C\ell_{0,n})$. Moreover, $\alpha=\nabla f/f$ solves \eqref{eq:Vekua_equation} if and only if $\Delta f=0$ in $\Omega$. More generally, $\alpha\in A_{\alpha,\beta}^p(\Omega,C\ell_{0,n})$ if and only if $D\alpha=|\alpha|^2 +\beta \alpha$.

Now, let us introduce the operator 
\begin{align}\label{eq:isomorphism}
S_{\alpha,\beta}:=I-T_{\Omega}[\alpha C+\beta I],
\end{align}
where $C$ is the usual Clifford conjugation operator, observe that  $S_{\alpha,\beta}$ is a continuous operator that transforms solutions of \eqref{eq:Vekua_equation} into left-monogenic functions in $\Omega$. 
The operator \eqref{eq:isomorphism} has been used previously in other works by the author; see, for instance, \cite{DelLeb2020, DelPor2019}. In particular, in \cite{DelPor2019} it was introduced the Vekua-Hardy spaces of solution of the main Vekua equation in the quaternionic case. In the absence of a Similarity Principle, this operator was very useful. 
The following results generalize \cite[Sec.\ 3]{DelLeb2020} in the complex case.
\begin{proposition}\label{lemma:transformationS}
Let $\alpha, \beta \in L^{\infty}(\Omega,C\ell_{0,n})$ and $w\in L^p(\Omega,C\ell_{0,n})$. Then
\begin{align*}
w\in A_{\alpha,\beta}^p(\Omega,C\ell_{0,n}) \Leftrightarrow S_{\alpha,\beta}[w] \in A^p(\Omega,C\ell_{0,n}). 
\end{align*}
\end{proposition}
\begin{proof}
The proof is analogous to \cite[Lem.\ 3.1]{DelLeb2020} in the context of the complex variable, just using that the Teodorescu transform $T_{\Omega}$ is a right inverse of the Moisil-Teodorescu operator $D$ in $\Omega$.
\end{proof}
\begin{proposition}\label{prop:Bergman-generalized-closed}
Let $1< p<\infty$ and let $\alpha, \beta\in L^{\infty}(\Omega,C\ell_{0,n})$. Then $A_{\alpha,\beta}^p(\Omega,C\ell_{0,n})$ is closed in $L^p(\Omega,C\ell_{0,n})$.
\end{proposition}
\begin{proof}
  Let $\left\{w_n\right\}\subseteq A_{\alpha,\beta}^p(\Omega,C\ell_{0,n})$ be a sequence such that $w_n\rightarrow w$ in $L^p(\Omega,C\ell_{0,n})$. By Proposition \ref{lemma:transformationS}, $S_{\alpha,\beta}[w_n]\in A^p(\Omega,C\ell_{0,n})$. Further, it converges to $S_{\alpha,\beta}[w]$ in $L^p(\Omega,C\ell_{0,n})$.
Indeed, by  continuity of the Teodorescu transform $T_{\Omega}$,
    we have:
\[ 
\|S_{\alpha}[w_n]-S_{\alpha,\beta}[w]\|_{L^p(\Omega)} \leq \|S_{\alpha,\beta}\| \|w_n-w\|_{L^p(\Omega)} \, ,
\] 
where $\|S_{\alpha,\beta}\|$ is the operator norm from $L^p(\D)$ to itself.
Consequently, by the closedness of the classical Bergman spaces $A^p(\Omega, C\ell_{0,n})$ in $L^p(\Omega,C\ell_{0,n})$ (see Lemma \ref{lemma:Hodge-monogenic}), we get that $S_{\alpha,\beta}[w]\in A^p(\Omega,C\ell_{0,n})$. Again by Proposition \ref{lemma:transformationS}, $w\in A_{\alpha,\beta}^p(\Omega,C\ell_{0,n})$, which finish the proof.
\end{proof}
\begin{corollary}\label{cor:reflexive-separable}
Let $1< p<\infty$ and let $\alpha, \beta\in L^{\infty}(\Omega,C\ell_{0,n})$. Then $A_{\alpha,\beta}^p(\Omega,C\ell_{0,n})$ is a reflexive and separable Banach space.
\end{corollary}
\begin{proof}
The proof is straightforward from Proposition \ref{prop:Bergman-generalized-closed} and the fact that subspaces of reflexive (separable) spaces are reflexive (separable) as well.
\end{proof}
\begin{remark}\label{rmk:Neumann-serie}
If $\|\alpha\|_{L^{\infty}}+\|\beta \|_{L^{\infty}}< 1/\|T_{\Omega}\|$, where $\|T_{\Omega}\|$ is the operator norm from $L^p(\Omega,C\ell_{0,n})$ to itself, then $S_{\alpha,\beta}$ has a bounded inverse operator in $L^p(\Omega,C\ell_{0,n})$ $1<p<\infty$, namely $S_{\alpha,\beta}^{-1}$. More precisely, the Neumann series \cite[Prop.\ 9.13]{Teschl2012}, \cite[Th.\ 1.3]{Kub2012} of the operator $S_{\alpha,\beta}^{-1}$ is given by
\begin{align}\label{eq:Neumann-series}
   S_{\alpha,\beta}^{-1}=\sum_{k=0}^{\infty} 
	 (T_{\Omega}(\alpha C+\beta I))^k.
\end{align}
\end{remark}
The hypothesis $\|T_{\Omega}[\alpha C+\beta I]\|<1$ is enough to guarantee the existence of the above Neumann series \eqref{eq:Neumann-series}, but we prefer to require a more restrictive condition $\|\alpha\|_{L^{\infty}}+\|\beta \|_{L^{\infty}}< 1/\|T_{\Omega}\|$ in terms of the functions $\alpha$ and $\beta$ of the Vekua equation.

It is not difficult to construct examples satisfying the hypothesis of Remark \ref{rmk:Neumann-serie}. It is well-known that the operator norm of the Teodorescu operator $T_{\Omega}$ from $L^p(\Omega, C\ell_{0,n})$ to itself \cite[Cor.\ 3.17]{GuSpr1997} $1<p<\infty$ satisfies 
\[
  \|T_{\Omega}\|\leq \text{diam }\Omega.
\]


The fact that $S_{\alpha,\beta}$ is an invertible operator according to Remark \ref{rmk:Neumann-serie} will be the key for many of the upcoming results. Now, we will see that $S_{\alpha,\beta}^{-1}$ sends left-monogenic function into solutions of the Vekua equation \eqref{eq:Vekua_equation} in $\Omega$.
\begin{proposition}\label{prop:inverse}
If $\|\alpha\|_{L^{\infty}}+\|\beta \|_{L^{\infty}}< 1/\|T_{\Omega}\|$, then the operator $S_{\alpha,\beta}^{-1}$ defined by the Neumann series \eqref{eq:Neumann-series} is a bounded operator and 
\[
S_{\alpha,\beta}^{-1}\colon A^p(\Omega,C\ell_{0,n}) \to A_{\alpha,\beta}^p(\Omega,C\ell_{0,n}).
\]
\end{proposition}
\begin{proof}
Let $u\in A^p(\Omega,C\ell_{0,n})$. Applying the Moisil-Teodorescu operator $D$ to the Neumann series \eqref{eq:Neumann-series} we have
\begin{align*}
D S_{\alpha,\beta}^{-1}[u]&=D\left(u+T_{\Omega}[\alpha \overline{u}+\beta u]+T_{\Omega}[\alpha \overline{T_{\Omega}[\alpha \overline{u}+\beta u]}+\beta T_{\Omega}[\alpha \overline{u}+\beta u]]+\cdots \right)\\
&=\alpha \overline{u}+\beta u + \alpha \overline{T_{\Omega}[\alpha \overline{u}+\beta u]} +\beta T_{\Omega}[\alpha \overline{u}+\beta u]+\cdots\\
&=\alpha \overline{S_{\alpha,\beta}^{-1}[u]}+\beta S_{\alpha,\beta}^{-1}[u].
\end{align*}
Therefore $S_{\alpha,\beta}^{-1}[u]\in A_{\alpha,\beta}^p(\Omega,C\ell_{0,n})$ as we desired. The boundedness of $S_{\alpha,\beta}^{-1}$ is a consequence of the Inverse Mapping Theorem \cite[Th.\ 1.1]{Kub2012}.
\end{proof}
The regularity of the solutions of the Vekua equation \eqref{eq:Vekua_equation} depends on the regularity of the coefficients $\alpha$ and $\beta$. At least under this Neumann series approach, if $\|\alpha\|_{L^{\infty}}+\|\beta \|_{L^{\infty}}< 1/\|T_{\Omega}\|$ and $w\in A_{\alpha,\beta}^p(\Omega,C\ell_{0,n})$, then there exists $u\in A^p(\Omega,C\ell_{0,n})$ such that $w=S_{\alpha,\beta}^{-1}[u]$. That is, $w=S_{\alpha,\beta}^{-1}[u]=\sum_{k=0}^{\infty} (T_{\Omega}(\alpha C+\beta I))^k[u]$. Due to $u$ is harmonic component-wise, then $u$ is smooth. Consequently, if $\alpha, \beta\in C^k(\Omega)$ $k\geq 0$, by the regularity of the Teodorescu transform $(T_{\Omega}(\alpha C+\beta I))^j[u]\in C^{k+1}(\Omega)$ for all $j\geq 1$. Thus, $w\in C^{k+1}(\Omega)$. Similarly, in the weak sense, if $\alpha, \beta\in W^{1,m}(\Omega)$ $1<m<\infty$, then $w\in W^{1, m+1}(\Omega)$.
\color{black}
\section{Hodge decomposition of $A_{\alpha,\beta}^2(\Omega,C\ell_{0,n})$}\label{sec:Hodge}
Let $M^{\alpha}$ be the right-hand side multiplication operator, $M^{\alpha}[w]:=w\, \alpha$, we will see in Theorem \ref{th:Hodge-decomposition} somehow this multiplication operator is associated with the orthogonal complement of the generalized Vekua space $A_{\alpha,\beta}^2(\Omega, C\ell_{0,n})$.

Let us define the orthogonal complement of $A_{\alpha,\beta}^2(\Omega,C\ell_{0,n})$ by
\[
A_{\alpha,\beta}^2(\Omega,C\ell_{0,n})^{\bot}=\{u\in L^2(\Omega,C\ell_{0,n}) \colon \langle w, u\rangle_{0,L^2}=0, \forall w\in A_{\alpha,\beta}^2(\Omega,C\ell_{0,n})\},
\]
with respect to the scalar product
\begin{equation}\label{eq:scalar-product}
   \langle u,v\rangle_{0,L^2}=\text{Sc }\int_{\Omega} \overline{u}v\, d\sigma, 
\end{equation}
for all $u, v\in L^2(\Omega,C\ell_{0,n})$. Observe that the above-mentioned scalar product is precisely the scalar part of the inner product defined in \eqref{eq:inner-product}. It is worth mentioning that both inner products generate the same topology. 
Unlike the inner product \eqref{eq:inner-product}, we can easily see that the scalar product \eqref{eq:scalar-product} is also $\mathbb{R}$-linear in the first argument on the right and left-hand side.
In summary, thanks to Proposition \ref{prop:Bergman-generalized-closed} we have that $A_{\alpha,\beta}^2(\Omega,C\ell_{0,n})$ is a Hilbert subspace over $\mathbb{R}$ under the scalar product \eqref{eq:scalar-product}.
An important property of the inner product \eqref{eq:inner-product} is the fact that $\langle u, v\rangle_{0,L^2}=\langle v, u\rangle_{0,L^2}$, for all $u,v\in L^2(\Omega,C\ell_{0,n})$.
We will use some well-known facts about the Hodge decomposition of the Bergman space \eqref{eq:Hodge-1} and the following factorization 
\begin{align}\label{eq:factorization-complete}
   \nonumber (D-\alpha C-\beta)(D-M^{\alpha}C-\overline{\beta})h&=(D-\alpha C-\beta)(Dh-\overline{h}\alpha -\overline{\beta} h)\\
   \nonumber  &=-\Delta h-D( \overline{h}\alpha)-D(\overline{\beta} h)+\alpha(\overline{h}D+\overline{\alpha}h+\overline{h}\beta)\\
   &-\beta (Dh-\overline{h}\alpha - \overline{\beta}h )  
\end{align}
In particular, if $h=h_0$ is a scalar-valued function, then taking the scalar part of the above factorization \eqref{eq:factorization-complete} we get 
\begin{equation}\label{eq:factorization-previous}
    \left(-\Delta+|\alpha|^2+|\beta|^2-\text{Sc }(D\alpha+D\overline{\beta}-2\alpha\beta)\right)h_0=\text{Sc } (D-\alpha C-\beta)(D-M^{\alpha}C-\overline{\beta})h_0.
\end{equation}
It is well-known that for every $a, b\in C\ell_{0,n}$, the scalar part of the product is given by $\text{Sc }(\overline{a}b)=\text{Sc }(a\overline{b})=a\cdot b$, where the scalar product is taken as the one for vectors $a,b\in \mathbb{R}^{2^n}$. Therefore, our factorization \eqref{eq:factorization-previous} can be rewritten as
\begin{equation}\label{eq:factorization}
    \left(-\Delta+|\alpha|^2+|\beta|^2+\div \vec \alpha-\div \vec \beta+2\alpha\cdot \overline{\beta}\right)h_0=\text{Sc } (D-\alpha C-\beta)(D-M^{\alpha}C-\overline{\beta})h_0,
\end{equation}
where $\vec \alpha=\text{Vec }\alpha$ and $\vec \beta=\text{Vec }\beta$.
Now, let us come back to the particular case when $\alpha=\nabla f/f$ and $\beta=\nabla g/g$, with $f, g$ be $W^{1,\infty}(\Omega)$ non-vanishing scalar functions, then the right-hand side of \eqref{eq:factorization} reduces to the Schrödinger equation
\begin{equation}\label{eq:Schrodinger-1}
   \left(-\Delta+\dfrac{\Delta f}{f}-\dfrac{\Delta g}{g}+2\dfrac{\nabla g}{g}\cdot \left(\dfrac{\nabla g}{g}-\dfrac{\nabla f}{f}\right)\right)h_0=\text{Sc } (D-\alpha C-\beta)(D-M^{\alpha}C-\overline{\beta})h_0.
\end{equation}
Moreover, if $\alpha=\nabla f/f$ and $\beta\equiv 0$, then \eqref{eq:factorization} reduces to
\[
   \left(-\Delta+\dfrac{\Delta f}{f}\right)h_0=\text{Sc } (D-\alpha C-\beta)(D-M^{\alpha}C-\overline{\beta})h_0.
\]
which is turn is the second-order partial differential equation satisfied by the scalar part of solutions of the Vekua equation (see Corollary \ref{cor:conductividad-Schrodinger}).

\begin{proposition}\label{prop:empty}
Let $\Omega$ be a Lipschitz domain, and $\alpha, \beta\in W^{1,\infty}(\Omega, \mathbb{R})$. Then
\[
  A_{\alpha,\beta}^2(\Omega,C\ell_{0,n})\cap (D-M^{\alpha}C-\overline{\beta}) W_0^{1,2}(\Omega,\mathbb{R})=\{0\}.
\]
\end{proposition}
\begin{proof}
Let $w\in A_{\alpha,\beta}^2(\Omega,C\ell_{0,n})\cap(D-M^{\alpha}C-\overline{\beta}) W_0^{1,2}(\Omega,\mathbb{R})$. Then there exists $h\in W^{1,2}(\Omega,\mathbb{R})$ such that $w=(D-M^{\alpha}C-\overline{\beta})h\in A_{\alpha,\beta}^2(\Omega,C\ell_{0,n})$ with $h|_{\partial\Omega}=0.$ Consequently, 
\begin{align*}
     (D-\alpha C-\beta)(D-M^{\alpha}C-\overline{\beta})h=0, \quad \text{ in }\Omega.
\end{align*}
Using the factorization \eqref{eq:factorization}, we get that $h$ is a solution of the following Schr\"{o}dinger equation with zero vanishing trace
\[
\left(-\Delta+|\alpha|^2+|\beta|^2+\div \vec \alpha-\div \vec \beta+2\alpha\cdot \overline{\beta}\right)h=0, \quad \text{ in }\Omega.
\]
By the uniqueness of the solution of the Dirichlet problem, then $h\equiv 0$ as we desired.
\end{proof}
An alternative proof using the adjoint operator $S_{\alpha,\beta}^*$ is the following

\begin{proposition}\label{prop:empty-2}
Let $\alpha, \beta\in L^{\infty}(\Omega,C\ell_{0,n})$ such that $S_{\alpha,\beta}$ be an injective operator in $L^2(\Omega,C\ell_{0,n})$. Then
\[
  A_{\alpha,\beta}^2(\Omega,C\ell_{0,n})\cap (D-M^{\alpha}C-\overline{\beta}) W_0^{1,2}(\Omega,C\ell_{0,n})=\{0\}.
\]
\end{proposition}
\begin{proof}
First, using that $T_{\Omega}D=I$ on $W_0^{1,2}(\Omega,\mathcal{C}l_{0,n})$ it is easy to verify that 
\[
S_{\alpha,\beta}^*D W_0^{1,2}(\Omega,\mathcal{C}l_{0,n})=(I-T_{\Omega}[\alpha C+\beta I])^*DW_0^{1,2}(\Omega,\mathcal{C}l_{0,n})=(D-M^{\alpha}C-\overline{\beta}) W_0^{1,2}(\Omega,C\ell_{0,n}),
\]
under the scalar product \eqref{eq:scalar-product}. 

Let $w\in A_{\alpha,\beta}^2(\Omega,C\ell_{0,n})\cap (D-M^{\alpha}C-\overline{\beta}) W_0^{1,2}(\Omega,C\ell_{0,n})= A_{\alpha,\beta}^2(\Omega,\mathcal{C}l_{0,n})\cap S_{\alpha,\beta}^*D W_0^{1,2}(\Omega,\mathcal{C}l_{0,n})$. By Lemma \ref{lemma:Hodge-monogenic}, then $w=S_{\alpha,\beta}^*g$, with $g\in A^2(\Omega,\mathcal{C}l_{0,n})^{\bot}$. Since $A^2(\Omega,\mathcal{C}l_{0,n})$ is $S_{\alpha,\beta}S_{\alpha,\beta}^*$ invariant, we have that $A^2(\Omega,\mathcal{C}l_{0,n})^{\bot}$ is $S_{\alpha,\beta}S_{\alpha,\beta}^*$ invariant \cite[Lem.\ 1.6]{Kub2012}. Therefore, again by the Hodge decomposition provided by Lemma \ref{lemma:Hodge-monogenic} $S_{\alpha,\beta}w=S_{\alpha,\beta}S_{\alpha,\beta}^*g\in A^2(\Omega,\mathcal{C}l_{0,n})\cap A^2(\Omega,\mathcal{C}l_{0,n})^{\bot}=\{0\}$. Due to $S_{\alpha,\beta}$ is injective, we conclude that $w=0$. 
\end{proof}


\begin{theorem}\label{th:Hodge-decomposition}(\textbf{Hodge decomposition theorem})
Let $\Omega$ be a bounded Lipschitz domain, and let $\alpha, \beta\in L^{\infty}(\Omega,C\ell_{0,n})$.
Then the Hilbert space $L^2(\Omega,C\ell_{0,n})$ allows the orthogonal decomposition
\begin{align}\label{eq:Hodge-2}  L^2(\Omega,C\ell_{0,n})=A_{\alpha,\beta}^2(\Omega,C\ell_{0,n})\oplus (D-M^{\alpha}C-\overline{\beta}) W_0^{1,2}(\Omega,C\ell_{0,n}),
\end{align} 
under the scalar product \eqref{eq:scalar-product}.
\end{theorem}
\begin{proof}
It remains to prove that $(D-M^{\alpha}C-\overline{\beta}) W_0^{1,2}(\Omega,C\ell_{0,n})=A_{\alpha,\beta}^2(\Omega,C\ell_{0,n})^{\bot}$.
Let $h\in W_0^{1,2}(\Omega,C\ell_{0,n})$ and $w\in A_{\alpha,\beta}^2(\Omega,C\ell_{0,n})$. By the Gauss's theorem and from the fact that $h$ has vanishing boundary values, then
\begin{align}\label{eq:computations-orthogonal}
  \nonumber \langle w, (D-M^{\alpha}C-\overline{\beta})h \rangle_{0,L^2}&=\text{Sc }\int_{\Omega} \overline{w} \, (Dh-\overline{h}\alpha-\overline{\beta}h)\, d\sigma\\
  \nonumber &=\text{Sc }\int_{\Omega} -((\overline{w}D)h+\overline{w}\overline{h}\alpha+\overline{w}\overline{\beta}h)\, d\sigma+\text{Sc }\int_{\partial\Omega} \overline{w} \,\eta\,  h\, ds\\
    \nonumber &=\text{Sc }\int_{\Omega} -((\overline{w}D)h+w\overline{\alpha}h+\overline{w}\overline{\beta}h)\, d\sigma\\
    &=\text{Sc }\int_{\Omega} \overline{(D-\alpha C-\beta)w} \, h\, d\sigma,
\end{align}
where the third equality is justified by $\text{Sc }(\overline{w}\overline{h}\alpha)=\text{Sc }(w\overline{\alpha}h)$. Thus, 
\begin{equation*}
(D-M^{\alpha}C-\overline{\beta}) W_0^{1,2}(\Omega,C\ell_{0,n})\subset A_{\alpha,\beta}^2(\Omega,C\ell_{0,n})^{\bot}.
\end{equation*}
Now, we will prove that 
\begin{equation}\label{eq:contention-2}
(D-M^{\alpha}C-\overline{\beta})C_{c}^{\infty}(\Omega,C\ell_{0,n})^{\bot}\subset A_{\alpha,\beta}^2(\Omega,C\ell_{0,n}).
\end{equation}

Let $w \in (D-M^{\alpha}C-\overline{\beta})C_{c}^{\infty}(\Omega,C\ell_{0,n})^{\bot}$, then
\[
0=\nonumber \langle w, (D-M^{\alpha}C-\overline{\beta})g \rangle_{0,L^2},
\]
for all $g\in C_{c}^{\infty}(\Omega,C\ell_{0,n})$. Using again \eqref{eq:computations-orthogonal}, we obtain
\begin{align*}
0=\text{Sc }\int_{\Omega} \overline{(D-\alpha C-\beta)w} \, g\, d\sigma,
\end{align*}
for all $g\in C_{c}^{\infty}(\Omega,C\ell_{0,n})$. That is, $w$ is a weak solution of the Vekua equation \eqref{eq:Vekua_equation}, which implies \eqref{eq:contention-2}. Applying the fact that $\overline{C_{c}^{\infty}(\Omega,C\ell_{0,n})}=W^{1,2}(\Omega,C\ell_{0,n})$ in \eqref{eq:contention-2} we readily obtain that $A_{\alpha,\beta}^2(\Omega,C\ell_{0,n})^{\bot}\subset (D-M^{\alpha}C-\overline{\beta}) W_0^{1,2}(\Omega,C\ell_{0,n})$, which completes the proof.
\end{proof}

Notice that the following result generalizes the Hodge decompositions obtained for $\beta$ be a paravector in $C\ell_{0,n}$ in \cite{Sprossig1995, Sprossig2001} and for $\beta=D p/p$ with $p\in C^{\infty}(\Omega,\R)$ be a positive function in \cite{Sprossig1993}.  
\begin{corollary}
In the case $\alpha\equiv 0$. The Hodge decomposition \eqref{eq:Hodge-2} of Theorem \ref{th:Hodge-decomposition} reduces to
\begin{align*}
  L^2(\Omega,C\ell_{0,n})=A_{0,\beta}^2(\Omega,C\ell_{0,n})\oplus (D-\overline{\beta}) W_0^{1,2}(\Omega,C\ell_{0,n}), 
\end{align*}
with respect to both inner products \eqref{eq:inner-product} and \eqref{eq:scalar-product}, respectively.
\end{corollary}
\begin{proof}
The proof is immediate just by changing the inner product \eqref{eq:scalar-product} by \eqref{eq:inner-product} in the computations made in \eqref{eq:computations-orthogonal}.
\end{proof}
\begin{remark}
By \eqref{eq:computations-orthogonal}, we can observe that 
\[
\langle w, (D-M^{\alpha}C-\overline{\beta})h \rangle_{0,L^2}=\langle (D-\alpha C-\beta)w, h \rangle_{0,L^2},\quad \forall h\in W_0^{1,2}(\Omega,C\ell_{0,n}).
\]
Thus, we can consider the operator $D-M^{\alpha}C-\overline{\beta}$ as the adjoint operator of the Vekua operator $D-\alpha C-\beta$, when we are restricted to the vanishing trace Sobolev space $W_0^{1,2}(\Omega,C\ell_{0,n})$. According to the interrelationship that exists between these two operators. We can also consider the subspace of $L^2$ solutions of 
\begin{equation}\label{eq:Vekua-adjoint}
    Dw=\overline{w}\alpha+\overline{\beta}w, \quad \text{ in }\Omega.
\end{equation}
Analogously to the previous analysis made in Theorem \ref{th:Hodge-decomposition}, we will have the next orthogonal decomposition
 \begin{align}\label{eq:Hodge-adjoint}  L^2(\Omega,C\ell_{0,n})=\text{Ker }(D-M^{\alpha}C
 -\overline{\beta})\cap L^2(\Omega,C\ell_{0,n})\oplus (D-\alpha C-\beta) W_0^{1,2}(\Omega,C\ell_{0,n}),
\end{align} 
under the scalar product \eqref{eq:scalar-product}. If we consider the multiplication of the operators in the opposite order, then the associated Schrödinger equation is
   \begin{equation}\label{eq:factorization-adjoint}
    \left(-\Delta+|\alpha|^2+|\beta|^2+\div \vec \alpha+\div \vec \beta+2\alpha\cdot \beta\right)h_0=\text{Sc } (D-M^{\alpha}C-\overline{\beta})(D-\alpha C-\beta)h_0,
\end{equation}
which differs from the factorization obtained in \eqref{eq:factorization}. In particular, if $\alpha=\nabla f/f$ and $\beta=\nabla g/g$, then \eqref{eq:factorization-adjoint} becomes
\[
   \left(-\Delta+\dfrac{\Delta f}{f}+\dfrac{\Delta g}{g}+2\dfrac{\nabla g}{g}\cdot \dfrac{\nabla f}{f}\right)h_0=\text{Sc } (D-M^{\alpha}C-\overline{\beta})(D-\alpha C-\beta)h_0.
\]
\end{remark}

\section{The Vekua projection}\label{sec:Bergman-projection}
Recall that $A_{\alpha,\beta}^2(\Omega)\subseteq L^2(\Omega)$ is a Hilbert subspace over $\mathbb{R}$ with respect to the scalar product \eqref{eq:scalar-product}.
Consequently, there exist ortho-projections
\begin{align}\label{eq:ortho-projections}
 \nonumber  P_{\alpha,\beta} &\colon L^2(\Omega,C\ell_{0,n}) \to A_{\alpha,\beta}^2(\Omega,C\ell_{0,n}),\\
   Q_{\alpha,\beta} &\colon L^2(\Omega,C\ell_{0,n}) \to A_{\alpha,\beta}^2(\Omega,C\ell_{0,n})^{\bot},
\end{align}
such that $P_{\alpha,\beta}+Q_{\alpha,\beta}=I$ in $L^2(\Omega,C\ell_{0,n})$. We call $P_{\alpha,\beta}$ the \textit{Vekua projection}. 

The first part of the following Proposition \ref{prop:invariance} is called the \textit{invariance formula} for the generalized Vekua space $A_{\alpha,\beta}^2(\Omega, C\ell_{0,n})$ and it is an immediate consequence of the orthogonality of the projections $P_{\alpha,\beta}$
and $Q_{\alpha,\beta}$:

\begin{proposition}\label{prop:invariance} The following statements hold
\begin{enumerate}
\item[(i)] If $h\in L^2(\Omega,C\ell_{0,n})$ and $w\in A_{\alpha,\beta}^2(\Omega,C\ell_{0,n})$, then 
\begin{align}\label{eq:invariance_formula}
\langle w,h\rangle_{0,L^2}=\langle w, P_{\alpha,\beta}[h]\rangle_{0,L^2}
\end{align}
\item [(ii)] $P_{\alpha,\beta}$ is self-adjoint in $L^2(\Omega,C\ell_{0,n})$.
\end{enumerate}
\end{proposition}
\begin{proof}
Let $h\in L^2(\Omega,C\ell_{0,n})$ and $w\in A_{\alpha,\beta}^2(\Omega,C\ell_{0,n})$, then
\begin{align*}
   \langle w,h\rangle_{0,L^2}=\langle w, P_{\alpha,\beta}[h]+Q_{\alpha,\beta}[h]\rangle_{0,L^2}=\langle w, P_{\alpha,\beta}[h]\rangle_{0,L^2}.
\end{align*}
Part $(ii)$ is a straightforward consequence of \eqref{eq:invariance_formula}.
\end{proof}
When $\alpha\equiv 0$ and $\beta\equiv 0$, \eqref{eq:invariance_formula} reduces to the well-known invariance formula for Bergman-type spaces of monogenic functions.
\begin{proposition}\label{prop:dual-space}
Let $\alpha, \beta \in L^{\infty}(\Omega,C\ell_{0,n})$. Then the dual space of $A_{\alpha,\beta}^2(\Omega,C\ell_{0,n})$ is isomorphic to itself.
\end{proposition}
\begin{proof}
The proof is analogous to \cite[Prop.\ 3.8]{DelLeb2020}. Due to the Clifford generalization of the Hahn Banach extension theorem \cite{BDS1982}, the fact that $L^2(\Omega, C\ell_{0,n})^*$ is isomorphic to $L^2(\Omega, C\ell_{0,n})$ (see \cite[Th.\ 10.1, 10.3]{GuHaSpr2008} for the quaternionic version), and the invariance formula \eqref{eq:invariance_formula}.
\end{proof}

Since $S_{\alpha,\beta}$ is a bounded operator from $A_{\alpha,\beta}^2(\Omega,C\ell_{0,n})$ to $A^2(\Omega,C\ell_{0,n})$, we have that $S_{\alpha,\beta}^*$
is the unique operator satisfying
\begin{align*}
   \langle S_{\alpha,\beta}[w], h\rangle_{0,L^2}=\langle w, S_{\alpha,\beta}^*[h]\rangle_{0,L^2},
\end{align*}
for every $w\in A_{\alpha,\beta}^2(\Omega,C\ell_{0,n})$ and $h\in A^2(\Omega,C\ell_{0,n})$. Moreover, 
\begin{align}\label{eq:adjoint}
S_{\alpha,\beta}^*\colon A^2(\Omega,C\ell_{0,n})\to A_{\alpha,\beta}^2(\Omega,C\ell_{0,n}), 
\end{align}
is a bounded operator and $\|S_{\alpha,\beta}^*\|=\|S_{\alpha,\beta}\|$, see for instance \cite[Sec.\ 1.5]{Kub2012}.


Recall the conditions analyzed in Remark \ref{rmk:Neumann-serie} to guarantee that $S_{\alpha,\beta}$ be an invertible operator. By the relation $\text{Ker }(T^*)=\text{Im}(T)^{\bot}$ \cite[Lemma\ 1.4]{Kub2012}, if $S_{\alpha,\beta}$ is invertible in $L^2(\Omega, C\ell_{0,n})$, then $S_{\alpha,\beta}^*$ is also invertible. Furthermore, $(S_{\alpha,\beta}^*)^{-1}=(S_{\alpha,\beta}^{-1})^*$. As a consequence of Proposition \ref{prop:inverse}, we obtain that
\begin{align}\label{eq:adjoint-inverse}
 (S_{\alpha,\beta}^*)^{-1}\colon A_{\alpha,\beta}^2(\Omega,C\ell_{0,n}) \to A^2(\Omega,C\ell_{0,n}).
\end{align}
Notice that
\begin{equation}\label{eq:orthogonal-complement}
A_{\alpha,\beta}^2(\Omega,C\ell_{0,n})^{\bot}=S_{\alpha,\beta}^* \, DW_0^{1,2}(\Omega,C\ell_{0,n})=S_{\alpha,\beta}^*( A^2(\Omega,C\ell_{0,n})^{\bot}),
\end{equation}
where the first equality comes from the proof of Proposition \ref{prop:empty-2} and from the characterization of the orthogonal complement provided by Theorem \ref{th:Hodge-decomposition}. Meanwhile, for the second equality, we used Lemma \ref{lemma:Hodge-monogenic}.

Let us define 
\begin{align*}
P_{\alpha,\beta}^{\prime}&:=S_{\alpha,\beta}^*\, P_{\Omega} \,(S_{\alpha,\beta}^*)^{-1}\colon L^2(\Omega,C\ell_{0,n}) \to A_{\alpha,\beta}^2(\Omega,C\ell_{0,n}),\\
Q_{\alpha,\beta}^{\prime}&:=S_{\alpha,\beta}^*\, Q_{\Omega}\, (S_{\alpha,\beta}^*)^{-1} \colon L^2(\Omega,C\ell_{0,n}) \to A_{\alpha,\beta}^2(\Omega,C\ell_{0,n})^{\bot},
\end{align*}
where $P_{\Omega}$ is the usual Bergman projection defined in 
\eqref{eq:Bergman_projection} and $Q_{\Omega}$ its corresponding ortho-projection. Notice that the fact that the images of $P_{\alpha,\beta}^{\prime}$ and $Q_{\alpha,\beta}^{\prime}$ belongs to $A_{\alpha,\beta}^2(\Omega,C\ell_{0,n})$ and $A_{\alpha,\beta}^2(\Omega,C\ell_{0,n})^{\bot}$ is justified by \eqref{eq:adjoint} and \eqref{eq:orthogonal-complement}, respectively.
\begin{theorem}\label{theo:Bergman-projection}
Let $\alpha, \beta\in L^{\infty}(\Omega,C\ell_{0,n})$ be such that $S_{\alpha,\beta}$ is an invertible operator in $L^2(\Omega, C\ell_{0,n})$.  
The Vekua projection $P_{\alpha,\beta}\colon L^2(\Omega,C\ell_{0,n}) \to A_{\alpha,\beta}^2(\Omega,C\ell_{0,n})$ is given by 
\begin{align}\label{eq:projection1}
    P_{\alpha,\beta}=P_{\alpha,\beta}^{\prime}=S_{\alpha,\beta}^*\, P_{\Omega} \,(S_{\alpha,\beta}^*)^{-1}.
\end{align}
Moreover, $P_{\alpha,\beta}$ is a bounded operator and $Q_{\alpha,\beta}=Q_{\alpha,\beta}^{\prime}$.
\end{theorem}

\begin{proof}
It is easy to verify that $(P_{\alpha,\beta}^{\prime})^2=P_{\alpha,\beta}^{\prime}$, precisely because of $P_{\Omega}^2=P_{\Omega}$. 
We will see that $P_{\alpha,\beta}^{\prime}$ and $Q_{\alpha,\beta}^{\prime}$ are ortho-projections. Indeed, let $u,v\in L^2(\Omega,C\ell_{0,n})$
\begin{align*}
   \langle P_{\alpha,\beta}^{\prime}[u],Q_{\alpha,\beta}^{\prime}[v]\rangle_{0,L^2}&=\langle S_{\alpha,\beta}^*P_{\Omega} (S_{\alpha,\beta}^*)^{-1}[u],S_{\alpha,\beta}^*Q_{\Omega} (S_{\alpha,\beta}^*)^{-1}[v]\rangle_{0,L^2}\\
	&=\langle P_{\Omega} (S_{\alpha,\beta}^*)^{-1}[u], S_{\alpha,\beta}S_{\alpha,\beta}^*Q_{\Omega} (S_{\alpha,\beta}^*)^{-1}[v]\rangle_{0,L^2}=0,
\end{align*}
where the last equality comes from the facts that $A^2(\Omega,C\ell_{0,n})^{\bot}$ is $S_{\alpha,\beta}S_{\alpha,\beta}^*$ invariant and $P_{\Omega}$ and $Q_{\Omega}$ are ortho-projections as well. 

Now, we will prove that $\text{Ker } P_{\alpha,\beta}^{\prime}=A_{\alpha,\beta}^2(\Omega,C\ell_{0,n})^{\bot}$.
Let $w\in \text{Ker } P_{\alpha,\beta}^{\prime}$, then $P_{\alpha,\beta}^{\prime}[w]=S_{\alpha,\beta}^*P_{\Omega} (S_{\alpha,\beta}^*)^{-1}[w]=0$, using that $S_{\alpha,\beta}^*$ is an injective operator, we have $P_{\Omega} (S_{\alpha,\beta}^*)^{-1}[w]=0$. That is, $(S_{\alpha,\beta}^*)^{-1}[w]\in \text{Ker } P_{\Omega}=D W_0^{1,2}(\Omega,C\ell_{0,n})$ (see Lemma \ref{lemma:Hodge-monogenic}). Therefore, 
\begin{equation}\label{eq:containment}
\text{Ker } P_{\alpha,\beta}^{\prime}\subset S_{\alpha,\beta}^*D W_0^{1,2}(\Omega,C\ell_{0,n})=A_{\alpha,\beta}^2(\Omega,C\ell_{0,n})^{\bot}.
\end{equation}
Reciprocally, let $w\in S_{\alpha,\beta}^*D W_0^{1,2}(\Omega,C\ell_{0,n})$, then $w=S_{\alpha,\beta}^*Dh$, with $h\in  W_0^{1,2}(\Omega,C\ell_{0,n})$. Thus, $P_{\alpha,\beta}^{\prime} [w]=S_{\alpha,\beta}^*P_{\Omega} (S_{\alpha,\beta}^*)^{-1}[w]=S_{\alpha,\beta}^*P_{\Omega} Dh=0$, because of $P_{\Omega}Q_{\Omega}=0$. Therefore, the inverse containment of \eqref{eq:containment} follows.

Now, we will prove that $\text{Im } P_{\alpha,\beta}^{\prime}=A_{\alpha,\beta}^2(\Omega,C\ell_{0,n})$. Let $w\in A_{\alpha,\beta}^2(\Omega,C\ell_{0,n})$. By 
\eqref{eq:adjoint-inverse}, $(S_{\alpha,\beta}^*)^{-1}[w]\in A^2(\Omega,C\ell_{0,n})$. By definition, $P_{\Omega}(S_{\alpha,\beta}^*)^{-1}[w]=(S_{\alpha,\beta}^*)^{-1}[w]$. Therefore, $w=P_{\alpha,\beta}^{\prime}[w]$, which implies $A_{\alpha,\beta}^2(\Omega,C\ell_{0,n})\subset \text{Im } P_{\alpha,\beta}^{\prime}$. The inverse containment is straightforward from \eqref{eq:adjoint}. By the uniqueness of the ortho-projections we obtain that $P_{\alpha,\beta}=P_{\alpha,\beta}^{\prime}$ and $Q_{\alpha,\beta}=Q_{\alpha,\beta}^{\prime}$.
To justify the boundedness, it is enough to prove that $(S_{\alpha}^*)^{-1}$ is bounded, which in turn is a consequence of the Inverse Mapping Theorem \cite[Th.\ 1.1]{Kub2012}.
\end{proof}
On the other hand, as consequence of Proposition \ref{lemma:transformationS} and due to $S_{\alpha,\beta}$ is an invertible operator we have
\begin{align}\label{eq:fixed}
\nonumber w\in A_{\alpha,\beta}^2(\Omega,C\ell_{0,n}) &\Leftrightarrow S_{\alpha,\beta}[w]\in A^2(\Omega,C\ell_{0,n}) \\
\nonumber &\Leftrightarrow P_{\Omega} S_{\alpha,\beta}[w]= S_{\alpha,\beta}[w]\\ 
&\Leftrightarrow w=S_{\alpha,\beta}^{-1} P_{\Omega} S_{\alpha,\beta}[w].
\end{align}
We can easily see that the following operators are other equivalent expressions for the ortho-projections $P_{\alpha,\beta}$ and $Q_{\alpha,\beta}$:
\begin{align*}
P_{\alpha,\beta}^{\prime\prime}&:=S_{\alpha,\beta}^{-1} \,P_{\Omega} \,S_{\alpha,\beta}\colon L^2(\Omega,C\ell_{0,n}) \to A_{\alpha,\beta}^2(\Omega,C\ell_{0,n}),\\
Q_{\alpha,\beta}^{\prime\prime}&:=S_{\alpha,\beta}^{-1}\, Q_{\Omega} \,S_{\alpha,\beta}\colon L^2(\Omega,C\ell_{0,n}) \to A_{\alpha,\beta}^2(\Omega,C\ell_{0,n})^{\bot},
\end{align*} 
\begin{corollary}
Let $\alpha, \beta\in L^{\infty}(\Omega,C\ell_{0,n})$ be such that $S_{\alpha,\beta}$ is an invertible operator. Then the Vekua ortho-projections admit the following representations
\begin{align}\label{eq:projection2}
P_{\alpha,\beta}=P_{\alpha,\beta}^{\prime}=P_{\alpha,\beta}^{\prime\prime}=S_{\alpha,\beta}^{-1}\, P_{\Omega} \,S_{\alpha,\beta},\qquad Q_{\alpha,\beta}=Q_{\alpha,\beta}^{\prime}=Q_{\alpha,\beta}^{\prime\prime}=S_{\alpha,\beta}^{-1} \,Q_{\Omega}\, S_{\alpha,\beta}
\end{align}
\end{corollary}
\begin{proof}
By Theorem \ref{theo:Bergman-projection}, $P_{\alpha,\beta}=P_{\alpha,\beta}^{\prime}$. By Proposition \ref{prop:invariance} $(ii)$, we have that the Vekua projection $P_{\alpha,\beta}$ is self-adjoint in $L^2(\Omega,C\ell_{0,n})$, we have $P_{\alpha,\beta}=(P_{\alpha,\beta}^{\prime})^*=S_{\alpha,\beta}^{-1} P_{\Omega} S_{\alpha,\beta}$.
\end{proof}


\subsection{The Vekua reproducing kernel}\label{sec:kernel}
Recall that $A^2(\Omega,C\ell_{0,n})$ possess a reproducing kernel, see for instance \cite{Legatiuk2018}. That is, there exists $K\in A^2(\Omega,C\ell_{0,n})$ such that
\[
w(\vec x)=\langle  K(\cdot, \vec x), w(\cdot)\rangle_{L^2},\qquad \forall  w\in A^2(\Omega,C\ell_{0,n}).
\]
Moreover,
\begin{equation}\label{eq:kernel-1}
   P_{\Omega}[w](\vec x)=\langle  K(\cdot, \vec x), w(\cdot)\rangle_{L^2},\qquad \forall  w\in L^2(\Omega,C\ell_{0,n}),
\end{equation}
where $P_{\Omega}$ is the ortho-projection in the monogenic case defined in \eqref{eq:Bergman_projection}.

The Vekua projection and the existence of a reproducing kernel is trivial when $\alpha\equiv 0$ and $\beta=\nabla g/g$. Because, in this case the Vekua operators is factorized as follows \cite{Sprossig1993}
\[
  (D-\nabla g/g)=gDg^{-1}.
\]
Consequently, the Vekua projection is trivial
\begin{equation*}
   P_{\beta}=gPg^{-1}, \qquad    P_{\beta}[w](\vec x)=g \langle  K(\cdot, \vec x), g^{-1}w(\cdot)\rangle_{L^2}
   =g  \langle  g^{-1}K(\cdot, \vec x), w(\cdot)\rangle_{L^2}.
\end{equation*}

Nevertheless, for the case $\alpha\ne 0$, we will show by a contradiction argument that $A_{\alpha,\beta}^2(\Omega,C\ell_{0,n})$ does not inherit the reproductive property \eqref{eq:kernel-1}. Let us suppose that $A_{\alpha,\beta}^2(\Omega,C\ell_{0,n})$ also possesses a reproducing kernel $K_{\alpha,\beta}\in A_{\alpha,\beta}^2(\Omega,C\ell_{0,n})$ such that 
\begin{equation*}\label{eq:kernel-2}
   P_{\alpha,\beta}[w](\vec x)=\langle  K_{\alpha,\beta}(\cdot, \vec x), w(\cdot)\rangle_{L^2},\qquad \forall w\in L^2(\Omega,C\ell_{0,n}).
\end{equation*}
Thus, for every $w\in A_{\alpha,\beta}^2(\Omega,C\ell_{0,n})$ and $a\in C\ell_{0,n}$, we have
\begin{align*}
    P_{\alpha,\beta}[wa]=\langle  K_{\alpha,\beta}(\cdot, \vec x), w(\cdot)a\rangle_{L^2}=\langle  K_{\alpha,\beta}(\cdot, \vec x), w(\cdot)\rangle_{L^2}a=wa,
\end{align*}
which implies that $wa\in A_{\alpha,\beta}^2(\Omega,C\ell_{0,n})$, contradicting the fact that $A_{\alpha,\beta}^2(\Omega,C\ell_{0,n})$ is not a subspace over $C\ell_{0,n}$.


After the above discussion, we can still consider some reproducing kernels component-wise under the scalar product \eqref{eq:scalar-product}. 

Let $\mathcal{P}$ be the power set of $\{1,2, \cdots, n\}$ and recall the basic decomposition of $C\ell_{0,n}$-valued function as $w=\sum_{A} w_A\,e_A$ for all $A\in \mathcal{P}$, where $e_A=e_{i_1}e_{i_2}\cdots e_{i_k}$ and $A=\{i_1, i_2, \cdots i_k\}$, $1\leq i_1<i_2<\cdots <i_k\leq n$.

\begin{proposition}\label{prop:existence}
If $\|\alpha\|_{L^{\infty}}+\|\beta \|_{L^{\infty}}< 1/\|T_{\Omega}\|$, with $\|T_{\Omega}\|$ the norm operator from $L^{\infty}(\Omega)$ to itselt, then there exists $K_{\alpha,\beta}^A(\cdot, \vec x)\in A_{\alpha,\beta}^2(\Omega,C\ell_{0,n})$ such that for all $w=\sum_{A} w_A\, e_A\in A_{\alpha,\beta}^2(\Omega,C\ell_{0,n})$
\begin{equation}\label{eq:reproducing-property}
w_A(\vec x)=\langle K_{\alpha,\beta}^A(\cdot, \vec x), w(\cdot) \rangle_{0,L^2}, \qquad \forall A\in \mathcal{P}.         
\end{equation}
\end{proposition}
\begin{proof}
Let $K\subseteq \Omega$ be a compact and $w\in A_{\alpha,\beta}^2(\Omega,C\ell_{0,n})$. 
By \cite[Th.\ 3.14]{GuSpr1997} there exists $C_1>0$ such that the Teodorescu operator $T_K$ considered over $K$ satisfies
\[
   \max_{x\in K}|T_K[w](x)|\leq C_1 \|w\|_{L^p(K)}\leq C \,\text{Vol }(K)\max_{x\in K}|w(x)|, \qquad n<p<\infty.
\]
Thus, $T_K$ is a bounded operator from $L^{\infty}(K)$ to itself, let us denote its norm operator as $\|T_K\|$. Since $\|T_K\|\leq \|T_{\Omega}\|$, then the restriction operator $S_{\alpha,\beta}^K:=I-T_{K}[\alpha C+\beta I]$ is a bounded in $L^{\infty}(K)$ with $\|T_{K}[\alpha C+\beta I]\|<1$. Therefore, $S_{\alpha,\beta}^K$ has a bounded inverse in $L^{\infty}(K)$ provided by the Neumann series \cite{Kub2012}
\[
    (S_{\alpha,\beta}^K)^{-1}=\sum_{m=0}^{\infty} T_K[\alpha C+\beta I]^m.
\]
Moreover,
\begin{align}\label{eq:bound-w}
   \max_{x\in K}|w(x)|=\max_{x\in K}|(S_{\alpha,\beta}^K)^{-1}\circ S_{\alpha,\beta}^K\,w(x)|\leq \|(S_{\alpha,\beta}^K)^{-1}\|\, \max_{x\in K}|S_{\alpha,\beta}^K[w](x)|,
\end{align}
where $\|(S_{\alpha,\beta}^K)^{-1}\|$ denotes the norm operator from $L^{\infty}(K)$ to itself. Due to $S_{\alpha,\beta}\colon A^2_{\alpha,\beta}(\Omega)\to A^2(\Omega)$, then $S_{\alpha,\beta}[w]$ is harmonic in $\Omega$, then there exists $C_2>0$ depending on $K$ such that
\begin{equation}\label{eq:elliptic}
\max_{x\in K}|S_{\alpha,\beta}^K[w](x)|\leq C_2 \|S_{\alpha,\beta}[w]\|_{L^2(\Omega)}.
\end{equation}
Combining \eqref{eq:bound-w} and \eqref{eq:elliptic}, we readily obtain that for all $A\in \mathcal{P}$
\[
\max_{x\in K}|w_A(x)|\leq \max_{x\in K}|w(x)|\leq C_3 \|w\|_{L^2(\Omega)}.
\]
Finally, by the Riesz representation theorem \cite{Davis1975} there exist reproducing \textit{Vekua kernels} $K_{\alpha,\beta}^A(\vec x,\vec y)$ ($\forall A\in \mathcal{P}$), such that $K_{\alpha,\beta}^A(\cdot, \vec x)\in A_{\alpha,\beta}^2(\Omega,C\ell_{0,n})$ satisfies the reproductive property \eqref{eq:reproducing-property}.
\end{proof}

An immediate consequence of Proposition \ref{prop:existence} is the following, for all $A,B\in \mathcal{P}$
\begin{align}\label{eq:symetry}
[K_{\alpha,\beta}^A]_B(\vec x,\vec y)=\langle K_{\alpha,\beta}^B(\cdot, \vec x), K_{\alpha,\beta}^A(\cdot,\vec y) \rangle_{0,L^2}=\langle  K_{\alpha,\beta}^A(\cdot,\vec y), K_{\alpha,\beta}^B(\cdot, \vec x) \rangle_{0,L^2}=[K_{\alpha,\beta}^B]_A(\vec y,\vec x).
\end{align}	
The following result generalizes \cite[Prop.\ 14]{Campos2020} in the complex case.
\begin{proposition}\label{prop:projector}
An equivalent expression for the Vekua projection $P_{\alpha,\beta}$ defined in \eqref{eq:ortho-projections} is
\begin{equation*}
P_{\alpha,\beta}[w](\vec x)=\bigintsss_{\Omega} \sum_B (-1)^{\frac{|B|(|B|+1)}{2}}\, e_B^2\,  K^B_{\alpha,\beta}(\vec x,\vec y) \, w_B(\vec y) \,d\sigma_{\vec y}, \qquad \forall w\in L^2(\Omega,C\ell_{0,n}). 
\end{equation*}
\end{proposition}
\begin{proof}
Let $w=\sum_{B}w_B\, e_B\in L^2(\Omega,C\ell_{0,n})$. By the component-wise reproducing property \eqref{eq:reproducing-property}, the symmetry property \eqref{eq:symetry} and from the fact that $\Sc(\overline{u}v)=\Sc(u\overline{v})$ for all $u,v\in C\ell_{0,n}$, we have
\begin{align*}
  P_{\alpha,\beta}[w](\vec x)&=\sum_A [P_{\alpha,\beta}[w]]_A(\vec x)\, e_A
  =\sum_A \langle K_{\alpha,\beta}^A(\cdot, \vec x), w(\cdot) \rangle_{0,L^2}\, e_A\\
  &=\sum_A \int_{\Omega} \Sc\left( K_{\alpha,\beta}^A (\vec  y,\vec x)\, \overline{w(y)}\right)\, d\sigma_{\vec y}\, e_A\\
  &=\sum_{A,B} \int_{\Omega} (-1)^{\frac{|B|(|B|+1)}{2}}\, e_B^2\,[K_{\alpha,\beta}^A]_B(\vec y, \vec x)\, w_B(\vec y)\,d\sigma_{\vec y}\, e_A\\
  &=\sum_{A,B} \int_{\Omega} (-1)^{\frac{|B|(|B|+1)}{2}}\, e_B^2\,[K_{\alpha,\beta}^B]_A(\vec x, \vec y)e_A\, w_B(\vec y)\, d\sigma_{\vec y},
\end{align*}
where the fourth chain of equalities is justified by the fact that $\Sc\left( K_{\alpha,\beta}^A (\vec y,\vec x)\, \overline{w(y)}\right)=\sum_B \Sc\left( K_{\alpha,\beta}^A (\vec  y,\vec x)\, w_B\overline{e_B}\right)=\sum_B [K_{\alpha,\beta}^A]_B(\vec y, \vec x)\, w_B(\vec y) e_B\overline{e_B}$ and $\overline{e_B}=(-1)^{\frac{|B|(|B|+1)}{2}}\, e_B$.
\end{proof}
In the next Proposition \ref{prop:several-properties} we can observe the importance of the construction of an orthonormal basis in the Hilbert space $A_{\alpha,\beta}^2(\Omega, C\ell_{0,n})$, because in that way it is possible to find a closed form for the component-wise Vekua kernel $K_{\alpha,\beta}^A$, for every $A\in \mathcal{P}$. 
\begin{proposition}\label{prop:several-properties}
The following statement follows
\begin{enumerate}
\item [(i)] $A_{\alpha,\beta}^2(\Omega,C\ell_{0,n})$ admits an orthonormal basis $\{e_n\}$. Moreover, $\left\{P_{\alpha,\beta}[e_n] \right\}$ is complete in $A_{\alpha,\beta}^2(\Omega,C\ell_{0,n})$.
\item[(ii)] An expression for the Vekua kernels $ K_{\alpha,\beta}^A$ is the following
\begin{align*}
    K_{\alpha,\beta}^A(\vec x,\vec y)=\sum_{n=0}^{\infty} [e_n(\vec x)]_A \, e_n(\vec y), \qquad \forall A\in \mathcal{P}.
\end{align*}
\end{enumerate}
\end{proposition}
\begin{proof}
The part $(i)$, is straightforward from the fact that $A_{\alpha,\beta}^2(\Omega,C\ell_{0,n})$ is separable (see Corollary \ref{cor:reflexive-separable}). Thus, $A_{\alpha,\beta}^2(\Omega,C\ell_{0,n})$ has an orthonormal basis \cite[Th.\ 5.11]{Brezis2011}, namely $\{e_n\}\subset A_{\alpha,\beta}^2(\Omega,C\ell_{0,n})$. 
Let $w\in A_{\alpha,\beta}^2(\Omega,C\ell_{0,n})$. Then $w=\sum_{n=0}^{\infty} a_n \,e_n$. Applying the Vekua projection $P_{\alpha,\beta}$, we have $w=P_{\alpha,\beta}[w] =\sum_{n=0}^{\infty} a_n \, P_{\alpha,\beta}[e_n]$. 

On the other hand, the part $(ii)$ is a consequence of the representation in Fourier series \cite[Cor.\ 5.10]{Brezis2011} of the Vekua kernels $K_{\alpha,\beta}^A\in A^2_{\alpha,\beta}(\Omega,C\ell_{0,n})$ and by the reproducing property \eqref{eq:reproducing-property}, we have
\[
K_{\alpha,\beta}^A(\vec x,\vec y)=\sum_{n=0}^{\infty}  \langle K_{\alpha,\beta}^A(\vec x,\cdot), e_n(\cdot)\rangle_{0,L^2}\, e_n(\vec y)=\sum_{n=0}^{\infty} [e_n(\vec x)]_A \, e_n(\vec y).\qedhere
\]
\end{proof}


\section{Conclusions}
 In this paper, we analyze the space of $L^p$ solutions of the Vekua equation in the framework of Clifford algebras. Moreover, we give a Hodge decomposition theorem for this class of generalized monogenic functions. Besides, we propose two different approaches for the construction of the Vekua projector, one through a composition of operators that interrelate the space of $L^p$ solutions of the Vekua equation with the $L^p$ space of monogenic functions, and the second approach is based on the construction of some component-wise reproducing kernel, in the classic Bergman's sense.
The author is aware that the construction of an orthonormal basis of the Vekua spaces $A^2_{\alpha,\beta}(\Omega)$ analyzed in this work is fundamental for obtaining explicit expressions of the Vekua projector $P_{\alpha,\beta}$ and the component-wise Vekua kernels $K_{\alpha,\beta}^A$ (see Proposition \ref{prop:projector} and \ref{prop:several-properties}).  
\section*{Acknowledgements}
We thank the anonymous reviewer for her/his useful comments and for pointing out the possible applications of the Vekua equations to boundary value problems or inverse problems. 
We also thank Pablo E. Moreira for the suggestion to use the modular property of the conjugation in Clifford algebras in the proof of Proposition \ref{prop:Beltrami equivalence}.

\end{document}